\documentclass[12pt]{amsart}
\usepackage{amsmath}
\usepackage{amsthm}
\usepackage{amsfonts}
\usepackage{amssymb}
\newtheorem{Theorem}{Theorem}[section]
\newtheorem{Proposition}[Theorem]{Proposition}
\newtheorem{Lemma}[Theorem]{Lemma}
\newtheorem{Corollary}[Theorem]{Corollary}
\theoremstyle{definition}
\newtheorem{Definition}{Definition}[section]
\theoremstyle{remark}

\numberwithin{equation}{section}

\newcommand{\Z}{{\mathbb Z}}
\newcommand{\R}{{\mathbb R}}
\newcommand{\C}{{\mathbb C}}

\begin{document}

\title[Reflection coefficients]{Generalized reflection coefficients}

\author{Christian Remling}

\address{Mathematics Department\\
University of Oklahoma\\
Norman, OK 73019}

\email{cremling@math.ou.edu}

\urladdr{www.math.ou.edu/$\sim$cremling}

\date{June 4, 2014}



\thanks{CR's work has been supported
by NSF grant DMS 1200553}
\begin{abstract}
I consider general reflection coefficients for arbitrary one-dimensional whole line
differential or difference operators of order $2$. These reflection coefficients
are semicontinuous functions of the operator:
their absolute value can only go down when limits are taken. This implies a corresponding semicontinuity
result for the absolutely continuous spectrum, which applies to a very large class of maps. In particular, we can consider shift maps
(thus recovering and generalizing a result of Last-Simon) and flows of the Toda and KdV hierarchies (this is new). Finally, I evaluate an
attempt at finding a similar general setup that gives the much stronger conclusion of reflectionless limit operators in more
specialized situations.
\end{abstract}
\maketitle
\section{Introduction}
A \textit{Herglotz function }is a holomorphic map $F:\C^+\to\C^+$ on the upper half plane $\C^+=\{ z\in\C : \textrm{Im}\, z>0\}$.
Given a pair of Herglotz functions $m_{\pm}$, we define generalized \textit{reflection coefficients }as follows:
\begin{equation}
\label{defR}
R_+(z) = \frac{\overline{m_+(z)}+m_-(z)}{m_+(z)+m_-(z)}, \quad\quad
R_-(z) = \frac{m_+(z)+\overline{m_-(z)}}{m_+(z)+m_-(z)}
\end{equation}
While, on a formal level, it is convenient to make Herglotz functions the basic objects, I'm really motivated
by the spectral theory of whole line differential or difference operators.
We thus like to think of $m_{\pm}$ as representing the half line Titchmarsh-Weyl $m$ functions of, for example, a Jacobi matrix
\[
(Ju)_n = a_n u_{n+1} + a_{n-1}u_{n-1} + b_n u_n
\]
on $\ell^2(\Z)$, or perhaps of a Schr\"odinger operator
\[
(Hu)(x) = -u''(x) + V(x) u(x)
\]
on $L^2(\R)$, or of more general operators. This abstract and general framework
provides the natural setting for what I want to do here.

It is in fact \textit{always }possible
to interpret a given pair of Herglotz functions as the half line $m$ functions
of a sufficiently general whole line spectral problem. Namely, the Herglotz functions are in one-to-one correspondence to trace-normed
(that is, $\textrm{tr }H(x)=1$) canonical systems
\[
Ju'(x) = zH(x) u(x)
\]
on the half line $0\le x < \infty$. Here, $J=\left( \begin{smallmatrix} 0 & -1 \\ 1 & 0 \end{smallmatrix}\right)$, and $H(x)\in\R^{2\times 2}$ is positive semidefinite.
See \cite{Win} for this fundamental result, which is based on important earlier work by de~Branges \cite{dB} and Potapov \cite{Pota}.

Thus \textit{pairs }of Herglotz functions can be parametrized by \textit{whole line }trace-normed canonical systems
(one half line for each Herglotz function). In particular, Jacobi and Schr\"odinger equations can be rewritten as canonical systems.
We will often find it convenient to write $H(x)$ (with $x\in\R$) to refer in this way to a pair of Herglotz functions. This is mainly a notational device.

As we will see, the reflection coefficients $R_{\pm}$ have very remarkable properties. Their use was first advertised by Rybkin in \cite{Ryb},
and Breuer-Ryckman-Simon \cite{BRSim} work with closely related general reflection coefficients.
Here, we want to use $R_{\pm}$ to analyze the absolutely continuous spectrum in completely general situations.

We will be especially interested in the set (more precisely, equivalence class of sets, defined up to sets of Lebesgue measure zero)
\[
\Sigma_{ac} = \{x\in\R : \textrm{\rm Im }m_{\pm}(x) > 0 \} .
\]
This is the set where the underlying whole line operator has
absolutely continuous spectrum of multiplicity $2$.

Another important notion in the study of the absolutely continuous spectrum is the following:
\begin{Definition}
We call the canonical system $H(x)$ \textit{reflectionless }on a Borel set $A\subset\R$ of positive Lebesgue measure if
\[
m_+(x) = -\overline{m_-(x)} \:\: \textrm{for almost every }x\in A .
\]
\end{Definition}
If $H(x)$ is reflectionless on $A$, then $A\subset\Sigma_{ac}$. Conversely, however, a general canonical system can not
be expected to be reflectionless on $\Sigma_{ac}$; quite on the contrary, reflectionless operators are very rare. These points are also illuminated
by Proposition \ref{P1.1} below.

Reflectionless operators are important because they can be thought of as the basic building blocks
of arbitrary operators with absolutely continuous spectrum; see \cite{Remac} for more on this.

We can read off both $\Sigma_{ac}$ and the set where $H$ is reflectionless from the reflection coefficients.
This immediately recommends them for further use if one is interested in the absolutely continuous spectrum.
Observe first of all that we may define
$R_{\pm}(z)$ by the formulae given for all $z\in\C^+$. Since Herglotz functions have boundary values
$F(x)\equiv\lim_{y\to 0+} F(x+iy)$ for (Lebesgue-)almost all $x\in\R$, we can also consider the almost everywhere defined functions
$R_{\pm}(x)$. The boundary value of the Herglotz function $m_++m_-$ cannot be equal to zero on a positive measure set,
so there are no problems with the denominators here.

It is also helpful to observe that $|R_+|=|R_-|$, so we can unambiguously write $|R|$ if we are only interested
in the \textit{absolute values }of the reflection coefficients. More generally, we will frequently use the simplified notation
$R$ to refer to an unspecified reflection coefficient ($R_+$ or $R_-$).
\begin{Proposition}
\label{P1.1}
(a) Up to sets of measure zero,
\[
\Sigma_{ac} = \{ x\in\R : |R(x)|< 1 \} .
\]
(b) $H$ is reflectionless precisely on the set where $R(x)=0$.
\end{Proposition}
\begin{proof}
This is immediate from the definition of $R$ and the fact that $\textrm{Im }m_{\pm}(x)\ge 0$.
\end{proof}
We could similarly recover $\Sigma^{(\sigma)}_{ac}$, defined as the set where
$\textrm{Im }m_{\sigma}>0$, for $\sigma=+$ or $-$: indeed, $x\in\Sigma^{(\sigma)}_{ac}$ precisely if
$R_{\sigma}(x)\not= 1$. However, we will make no use of these sets here.

The following is the main result of this paper; recall our notational convention to write
$H$ (thought of as the coefficient of a canonical system) to refer to a pair of Herglotz functions $m_{\pm}$.
\begin{Theorem}[Semicontinuity of $R(x)$ and $\Sigma_{ac}$]
\label{Tsemicon}
If $m^{(n)}_{\pm}(z)\to m_{\pm}(z)$ locally uniformly on $z\in\C^+$,
then the reflection coefficients satisfy
\[
\left| R(x; H)\right| \le \limsup_{n\to\infty} \left| R(x; H_n) \right|
\]
for almost every $x\in\R$.

In particular, if $A$ denotes the set where the $\limsup$ is $<1$, then $\Sigma_{ac}(H)\supset A$.
\end{Theorem}
It is obvious from the definition that the reflection coefficients are continuous off the real line, that is,
$R(z;H_n)\to R(z;H)$ locally uniformly on $z\in\C^+$. It is also clear that $R(x;H)$ for $x\in\R$ will not,
in general, be \textit{continuous }(rather
than just semicontinuous) in the second argument. For example, an arbitrary pair $m_{\pm}$ of Herglotz functions
can be approximated by Herglotz functions with purely singular measures, and then $R(x;H_n)=1$
for almost every $x\in\R$ and all $n$. Or, as will become clear in a moment (see Proposition \ref{P1.2} below),
we can have $|R(x; H_n)|\ge c>0$ and $R(x;H)=0$ on a positive measure set $x\in A$ even if the approximating
problems have purely absolutely continuous spectra. In fact, the results on reflectionless limit points \cite{Remac}
say that something of this sort will always happen in certain situations.

Theorem \ref{Tsemicon} deals with completely general sequences of operators $H_n$, and then the $\limsup$ could easily
equal one and we are not getting a non-trivial statement. It is therefore worth pointing out that in the situations
of greatest interest, $|R(x; H_n)|$ will actually be constant along the whole sequence, and then the
$\limsup$ just equals the reflection coefficient $|R(x; H_1)|$ of the first operator of the sequence.

Let me elaborate on this in more detail.
We are often interested in sequences that are generated by iterating a map on the associated operators.
Of particular interest are the shift map and flows of an integrable hierarchy such as the Toda (Jacobi operators)
or KdV (Schr\"odinger operators) hierarchies. These maps are induced by an associated transfer matrix,
and this will imply that they preserve $|R|$.

Let me state this more formally.
Let $T(z)$ be an entire, $SL(2,\C)$ valued
function with $T(x)\in SL(2,\R)$ for $x\in\R$.
If $m_{\pm}^{(1)}$ are the half line $m$ functions of the
canonical system $H_1(x)$, define two new functions $m_{\pm}^{(2)}$ on $\C^+$
as follows:
\begin{equation}
\label{defphi}
m_-^{(2)}(z) = T(z)m_-^{(1)}(z) , \quad\quad m_+^{(2)}(z) = IT(z)I\,m_+^{(1)}(z)
\end{equation}
Here $I= \left( \begin{smallmatrix} 1 & 0\\ 0 & -1 \end{smallmatrix} \right)$, and a matrix acts on a complex number
as a linear fractional transformation, that is,
\[
\begin{pmatrix} a & b \\ c & d \end{pmatrix} z = \frac{az+b}{cz+d} .
\]
(So $I$ implements multiplication by $-1$.)
The new functions $m_{\pm}^{(2)}$ are not guaranteed to be Herglotz functions, but if they are, we obtain a transformation
$m_{\pm}^{(1)}\mapsto m_{\pm}^{(2)}$.

\begin{Definition}
\label{D1.1}
We say that a transformation of canonical systems $H_1\mapsto H_2$ is of type TM (as in transfer matrix) if
the $m$ functions of the new system $H_2$ are obtained as in \eqref{defphi}, for some $T(z)$ as above.
\end{Definition}
This definition may not, at first sight, look especially natural. The next two Propositions will, I hope, show that it is useful here.
\begin{Proposition}
\label{P1.3}
The following maps are of type TM:
\begin{enumerate}
\item shift maps;
\item (time $t$ maps of) flows from the Toda hierarchy;
\item (time $t$ maps of) flows from the KdV hierarchy
\end{enumerate}
\end{Proposition}
By a \textit{shift map, }we of course mean a map that sends $H(x)$ to $H(x-a)$, for some $a\in\R$.
This includes (as restrictions) shifts $V(x)\mapsto V(x-a)$ of Schr\"odinger operators and shifts
$(a_n,b_n)\mapsto (a_{n-k},b_{n-k})$ of Jacobi matrices.

The list from Proposition \ref{P1.3} is, of course, not exhaustive. For example, we do not go beyond
polynomial dependence on $z$ here in the sense that either $T(z)$ itself is a polynomial (Jacobi shift)
or is obtained by solving a differential equation with coefficients that are polynomials in $z$.
\begin{proof}[Sketch of proof]
This follows from the fact that the $m$ functions can be represented by formulae of the type
$m_{\pm}(z) =\pm u_2(0,z)/u_1(0,z)$, where $u$ solves the canonical system (or other associated
equations, in cases (2), (3)) $Ju'=zHu$. If we now shift $H$ by $a$, then $u$ is updated by the transfer matrix
$T(a,z)$, which indeed leads to \eqref{defphi}.

We obtain the same conclusion for
Toda and KdV type flows because these also have associated $SL(2,\C)$/$SL(2,\R)$ cocycles that update solutions.
See, for example, \cite{GeHol1,GeHol2,Teschl} for a discussion of this.
\end{proof}
\begin{Proposition}
\label{P1.2}
If $H_1\mapsto H_2$ is of type TM, then
\[
\left| R(x;H_2)\right| = \left| R(x; H_1) \right|
\]
for almost all $x\in\R$.
\end{Proposition}
\begin{proof}
Since $T(z)$ is assumed to be continuous, \eqref{defphi} will also hold for almost all $z=x\in\R$. Then we
can obtain $R(x;H_2)$ directly from $R(x; H_1)$ by using formula \eqref{defR} for $z=x$ and updating
$m_{\pm}(x)$ according to \eqref{defphi}.
We can then verify by a direct computation that this will not change $|R(x)|$. More explicitly, if $T(x) =\left( \begin{smallmatrix}
a & b \\ c & d \end{smallmatrix} \right)$, then we find that
\[
R_+(x; H_2) = \frac{c m -d}{c\overline{m}-d}\, R_+(x; H_1) , \quad m = m_+^{(1)}(x) .
\]
\end{proof}
In the next Corollary, I'll focus on the
version for Jacobi operators, but analogous results hold in other settings, with the same proof.
\begin{Corollary}
\label{C1.1}
Suppose that $J_n\to J$ in the weak operator topology (equivalently, assume pointwise convergence of the
coefficients $a, b$). If $J_{n+1}$ is obtained from $J_n$ by:
\begin{enumerate}
\item a shift, or
\item a (time $t$ map of a) flow from the Toda hierarchy,
\end{enumerate}
then $\Sigma_{ac}(J)\supset \Sigma_{ac}(J_1)$.
\end{Corollary}
\begin{proof}
These transformations are of type TM by Proposition \ref{P1.3}.
By Proposition \ref{P1.2}, $|R(x; J_n)|$ is constant along the whole sequence, so the result is immediate from Theorem \ref{Tsemicon}
if we also recall that pointwise convergence of the coefficients of a Jacobi matrix
is equivalent to the locally uniform convergence of the $m$ functions.
\end{proof}

The pure shift version of the Corollary was obtained earlier by Last-Simon \cite{LS}, and this is now also
an immediate consequence of the fact that limit points under the shift map are reflectionless on
$\Sigma_{ac}(J_1)$ (see \cite{Remac} for this); the result for Toda flows is new. Here we can in particular
consider $\omega$~limit points $J=\lim\phi_{t_n}(J_1)$ under a fixed Toda flow $\phi_t$; of course, much
more general limits are possible, too.

If we only allow shifts and consistently shift in one direction only (say $(a_n,b_n)\mapsto (a_{n-k},b_{n-k})$ with $k\ge 0$),
then we have the much stronger statement available that all limit points are reflectionless
on $\Sigma_{ac}(J_1)$ \cite[Theorem 1.4]{Remac}. Such a statement cannot possibly hold in the
general setup of Theorem \ref{Tsemicon} or Corollary \ref{C1.1}; for example, we could have that $J_n=J_1$ for all $n\ge 1$.
We must impose an irreversibility condition on the evolution to have any right to expect such a result. It does not seem
ridiculous, however, to hope that if $J=\lim \phi_{t_n}(J_1)$ with $t_n\to\infty$ and $\phi$ a Toda type flow,
then $J$ will be reflectionless on $\Sigma_{ac}$. Kotani \cite{KotKdV} has such a result
for the KdV flow with special initial conditions. In this work, Kotani runs a localized (in energy) version of the proof from
\cite{Remac}, making use of the associated cocycle of the KdV flow.

Inspired by this and Corollary \ref{C1.1}, it is now tempting to ask what class of evolutions of this type lets one
run the original proof and obtain reflectionless limit points. I'll give in to this temptation in the final section of
this paper. The answer we obtain here will be satisfying and disappointing at the same time:
we essentially recover the class of shift maps, but now in the generality of canonical systems.

Theorem \ref{Tsemicon} is proved in Section 3, and Section 2 discusses
a criterion for convergence of Herglotz functions that will be a crucial ingredient to this proof.
\section{Convergence of Herglotz functions}
Our proof of Theorem \ref{Tsemicon} will be based on the following criterion for convergence of Herglotz functions.
Recall again that the boundary value $F(x)$ ($x\in\R$) of a Herglotz
function cannot be equal to zero on a positive measure set, so we can
certainly define $\ln F(x)$ almost everywhere. We will always use the value with $0\le \textrm{Im}\,\ln F(x)\le\pi$.

In fact, we will be working with (complex) logarithms a lot in this section and the next, and much ink can be saved
if at this point already we agree once and for all that
\begin{equation}
\label{convlog}
0\le \textrm{\rm Im}\, \ln w < 2\pi
\end{equation}
is the value we want whenever a logarithm makes an appearance somewhere.
\begin{Theorem}
\label{THergl}
Let $F_n, F$ be Herglotz functions, and fix $R>0$. Then $\ln F_n(x), \ln F(x) \in L^2(-R,R)$.
Moreover, if $F_n(z)\to F(z)$ locally uniformly on $z\in\C^+$, then
\[
\ln F_n(x)\to \ln F(x) \:\: \textrm{\rm weakly in } L^2(-R,R) .
\]
\end{Theorem}
The converse is also true and follows automatically from this, by compactness. More precisely,
if $F_n,F$ are Herglotz functions and $\ln F_n\to\ln F$ weakly in $L^2(A)$ for some bounded $A\subset\R$
of positive Lebesgue measure, then $F_n(z)\to F(z)$ locally uniformly on $\C^+$. Indeed, if this were
false, then $F_n(z)\to G(z)$ on a subsequence for some $G$ not identically equal to $F$, by normal families.
Since Herglotz functions are determined by their boundary values on a positive measure set, this
contradicts the Theorem. (To tell the truth, we really need a slightly more general version here because in this
argument, we could have $G\equiv a\in\R_{\infty}$.)
\begin{proof}
Any Herglotz function $F$ has a holomorphic logarithm $\ln F(z)$, which is a Herglotz function itself (recall \eqref{convlog}!).
Since $\textrm{Im}\, \ln F <\pi$, the representing measure of $\ln F$ is purely absolutely continuous; its density, up to a factor of $\pi$,
is called the \textit{Krein function }of $F$:
\[
\xi(x) = \frac{1}{\pi}\lim_{y\to 0+} \textrm{Im }\ln F(x+iy) .
\]
So we have that $0\le\xi\le 1$.

The Herglotz representation of $\ln F$ then reads
\begin{align}
\label{lnF}
\ln F(z) & = C + \int_{-\infty}^{\infty} \left( \frac{1}{t-z} - \frac{t}{t^2+1}\right)\xi(t)\, dt\\
\nonumber
& = C + \int_{-\infty}^{\infty} \frac{1+tz}{t-z}\, \zeta(t)\, dt ;
\end{align}
here we have introduced $\zeta(t)=\xi(t)/(t^2+1)$.

It is well known that locally uniform convergence
of Herglotz functions implies (in fact, is essentially equivalent to) weak $*$ convergence of the associated measures. See,
for example, \cite[Sect.\ 2]{Remac} for a more detailed discussion.
If now $F_n\to F$ locally uniformly, then also $\ln F_n(z)\to\ln F(z)$,
so we obtain that $\zeta_n\, dt\to \zeta\, dt$ in weak $*$ sense, as Borel measures on the one-point compactification
$\R_{\infty}$ of $\R$. The absence of a term linear in $z$ in the Herglotz representation of $\ln F$ implies
that our measures don't give weight to the added point $\infty$, so it is enough
to extend the integrations over $\R$. We then have that
\[
\int_{-\infty}^{\infty}  f(t)\zeta_n(t)\, dt \to \int_{-\infty}^{\infty} f(t)\zeta(t)\, dt
\]
for all test function $f$ that have a continuous extension to $\R_{\infty}$.

Our assignment is to analyze the sequence of functions
\[
\ln F_n(x)\equiv \lim_{y\to 0+}\ln F_n(x+iy) ,
\]
as $n\to\infty$.
We will do this by looking separately at the various contributions from \eqref{lnF}.

Since $C=\ln |F(i)|$, it is certainly clear that $C_n\to C$. Next, the weak $*$ convergence $\zeta_n\, dt\to\zeta\, dt$ together
with the uniform bound on the $\zeta_n$ give that
\[
\int_{|t|>R+1} \frac{1+tx}{t-x}\zeta_n(t)\, dt - \int_{|t|>R+1} \frac{1+tx}{t-x}\zeta(t)\, dt \to 0 ,
\]
uniformly on $|x|\le R$. So let's now focus on the integrals from the right-hand side of \eqref{lnF}, cut off at $|t|=R+1$. Then both summands
from the first line of \eqref{lnF}
become integrable separately, and
\[
\int_{-R-1}^{R+1} \frac{t}{t^2+1}\xi_n(t)\, dt \to \int_{-R-1}^{R+1} \frac{t}{t^2+1}\xi(t)\, dt ,
\]
by the weak $*$ convergence again. It remains to analyze
\[
f_n(x) \equiv \lim_{y\to 0+} \int_{-R-1}^{R+1} \frac{1}{t-x-iy}\, \xi_n(t)\, dt ,
\]
for $|x|<R$ and as $n\to\infty$. It is clear that $\textrm{Im }f_n = \pi\xi_n$ converges to $\textrm{Im }f=\pi\xi$ weakly on
$L^2(-R,R)$: this follows from the weak $*$ convergence $\zeta_n\, dx\to\zeta\, dx$ and the fact that $0\le \xi_n\le 1$.

The real part of $f_n$ can alternatively be computed as a Hilbert transform:
\begin{equation}
\label{2.1}
\textrm{\rm Re}\; f_n(x) = \lim_{y\to 0+} \int_{\{ t\in\R : |t-x|>y\} } \chi_A(t) \frac{\xi_n(t)}{t-x}\, dt \equiv (H(\chi_A\xi_n))(x)
\end{equation}
This holds for almost all $x\in (-R,R)$, and we have abbreviated $A=(-R-1,R+1)$. The Hilbert transform defines
a bounded operator on $L^2(\R)$, and it satisfies $H^*=-H$. This shows, first of all, that indeed $\ln F\in L^2(-R,R)$
for any Herglotz function $F$: since $\chi_A\xi\in L^2(\R)$, we have that $\textrm{Re}\, f\in L^2(-R,R)$ by \eqref{2.1}, and the other contributions
from \eqref{lnF} are bounded on $x\in (-R,R)$. Next, the desired weak convergence in $L^2(-R,R)$ also follows: write $B=(-R,R)$
for notational convenience, and let
$g\in L^2(B)$ be a test function. Then
\begin{align*}
\langle g, \textrm{Re}\, f_n \rangle_{L^2(B)} & = \langle \chi_B g, H(\chi_A\xi_n) \rangle_{L^2(\R)} = - \langle H(\chi_B g), \xi_n\rangle_{L^2(A)}\\
& \to - \langle H(\chi_B g), \xi \rangle_{L^2(A)} = \langle g, \textrm{Re}\, f \rangle_{L^2(B)} .
\end{align*}
Here, the convergence follows from the weak convergence $\xi_n\to\xi$ in $L^2(A)$ that was observed earlier, and the last equality is established
by reversing the first two steps.
\end{proof}
\section{Proof of Theorem \ref{Tsemicon}}
If $C<1$ and $L>0$ are given and we define
\[
B=\{ x\in [-L,L] :\limsup_{n\to\infty} |R(x;H_n)|< C \} ,
\]
then we can find an $A\subset B$ of almost full measure inside $B$ and $N\ge 1$,
such that $|R(x; H_n)|\le C$ for all $x\in A$ and $n\ge N$.
It thus suffices to show that $|R(x; H)|\le C$ almost everywhere on $A$ for such an $A$.

We consider the functions
\[
L^{(n)}_{\sigma}(x) = \ln \left( R_{\sigma}(x; H_n) - 1\right) ,
\]
for $x\in A$ (and $n\ge N$, $\sigma=\pm$).
The argument of the logarithm lies in the left half plane;
so $L^{(n)}$ has imaginary part in $(\pi/2, 3\pi/2)$ by \eqref{convlog}.
Since $R$ is bounded away from $1$, the $L^{(n)}$ are uniformly bounded on $A$. Thus they converge weakly in $L^2(A)$
on a suitable subsequence, which, for notational convenience, we assume to be the original sequence. The weak limits will take
values in the closed convex hull of the union of the ranges of the $L^{(n)}$. We now make use of the following elementary
fact; this will be proved after we have completed the proof of the theorem.
\begin{Lemma}
\label{L3.1}
Let $C<1$. Then $\{ \ln (z-1) : |z|\le C \}$ is a convex subset of $\C$.
\end{Lemma}
So we can now introduce two new functions $r_{\sigma}(x)$ on $x\in A$ to describe the weak limits, as follows:
\begin{equation}
\label{3.1}
\ln \left( R_{\sigma}(x; H_n) - 1\right) \xrightarrow{w} \ln \left( r_{\sigma}(x) - 1\right)
\end{equation}
By the Lemma, these may be required to take values in $|r|\le C$.

From the definition \eqref{defR}, we have that
\begin{equation}
\label{3.31}
R_{\sigma}-1 = -\frac{2i\,\textrm{\rm Im}\, m_{\sigma}}{m_++m_-} = 2i g\, \textrm{\rm Im}\, m_{\sigma} ;
\end{equation}
Here we have abbreviated $g=-1/(m_++m_-)$. Note that $g$ is a Herglotz function, too.
We would now like to take logarithms on both sides, but to be able to do this, we also need to make sure
that $\textrm{Im}\, m>0$ or, equivalently, $R\not = 1$. This is of course clear for $R=R(x;H_n)$, from the definition
of the set $A$, but we will also want to use these formulae for $R=R(x;H)$. This issue will be addressed in a moment.

Let's first use \eqref{3.31} for $R=R(x;H_n)$. As planned, take logarithms to obtain that
\begin{equation}
\label{3.3}
\ln \left( R_{\sigma}(x;H_n)-1\right) = \ln 2ig_n(x) + \ln\textrm{Im}\, m_{\sigma}^{(n)}(x) .
\end{equation}
The hypothesis of Theorem \ref{Tsemicon}
implies that also $g_n(z)\to g(z)$ locally uniformly. Thus $\ln 2ig_n(x)\xrightarrow{w} \ln 2ig(x)$ weakly in $L^2(A)$ by
Theorem \ref{THergl}. This shows that the weak limit
\begin{equation}
\label{defP}
P_{\sigma}(x) \equiv \textrm{\rm w--}\!\!\lim_{n\to\infty} \ln\textrm{Im}\, m^{(n)}_{\sigma}(x) ;
\end{equation}
exists, at least if we take it along the subsequence that makes \eqref{3.1} happen. We obtain the following limiting version of \eqref{3.3}:
\begin{equation}
\label{3.32}
\ln \left(r_{\sigma}(x)-1 \right) = \ln 2ig(x) + P_{\sigma}(x)
\end{equation}

We have the following fundamental estimate on $P_{\sigma}$.
\begin{Lemma}
\label{L3.2}
For almost every $x\in A$, we have that
\[
\textrm{\rm Im}\, m_{\sigma}(x) \ge e^{P_{\sigma}(x)} .
\]
\end{Lemma}
\begin{proof}[Proof of Lemma \ref{L3.2}]
Drop $\sigma$ in the notation, and let $I=(a,b)$ be an interval with
$\rho(\{a\} )=\rho(\{ b\} )=0$, where $\rho$ is the measure from the Herglotz representation of $m$. Then
\begin{align*}
\frac{\pi\rho(I)}{|I|} & = \lim_{n\to\infty} \frac{\pi\rho_n(I)}{|I|} \ge \limsup_{n\to\infty} \frac{1}{|I|}
\int_{I\cap A} \textrm{\rm Im}\, m_n(x)\, dx\\
& \ge \frac{|I\cap A|}{|I|}\lim_{n\to\infty} \exp \left( \frac{1}{|I\cap A|}\int_{I\cap A} \ln\textrm{\rm Im}\, m_n(x)\, dx \right) \\
& = \frac{|I\cap A|}{|I|} \exp \left( \frac{1}{|I\cap A|}\int_{I\cap A} P(x) \, dx\right) .
\end{align*}
We have used Jensen's inequality to pass to the second line, and the limit in this formula exists by \eqref{defP}.
Now take $I=(x-h,x+h)$, fix $x$ and send $h\to 0+$. For almost every $x\in A$,
the following statements will be true:
\[
\frac{|I\cap A|}{|I|}\to 1, \quad \frac{1}{|I|} \int_{I\cap A} P(t)\, dt \to P(x), \quad \frac{\pi \rho(I)}{|I|} \to \textrm{\rm Im}\, m(x)
\]
Thus, if we also avoid those $h>0$ for which $\rho(\{ x\pm h\})>0$, we obtain that $\textrm{Im}\, m(x)\ge e^{P(x)}$ for these
$x$, as claimed.
\end{proof}

In particular, this implies that $\textrm{Im}\, m_{\sigma}(x)>0$ for almost all $x\in A$, so we may indeed take
logarithms in \eqref{3.31} also for $R=R(x;H)$. Comparison of this formula with \eqref{3.32} then yields
\begin{equation}
\label{3.2}
\ln \left( R_{\sigma}(x; H) - 1\right) = \ln \left( r_{\sigma}(x) - 1\right) + \ln\textrm{\rm Im}\, m_{\sigma}(x) - P_{\sigma}(x) .
\end{equation}
(Note that it would of course have been too naive to expect, based on \eqref{3.1} perhaps, that $R(x;H)=r(x)$.)
We reformulate this one more time: For almost all $x\in A$, we have that
\begin{align}
\label{3.6a}
\textrm{\rm Im}\, \ln \left( R_{\sigma}(x)-1\right) & = \textrm{\rm Im}\, \ln \left( r_{\sigma}(x)-1\right) = \textrm{\rm Im}\,
\ln 2ig(x), \\
\label{3.6b}
\textrm{\rm Re}\, \ln \left( R_{\sigma}(x)-1\right) & \ge \textrm{\rm Re}\, \ln \left( r_{\sigma}(x)-1\right) .
\end{align}
To obtain this, we just take real and imaginary parts of \eqref{3.2}; the inequality in \eqref{3.6b} is obtained from Lemma \ref{L3.2}.

To finish the proof, we now discuss conditions \eqref{3.6a}, \eqref{3.6b} for fixed $x$. We know that $|r_+(x)|, |r_-(x)|\le C$,
and we wish to show that then $|R_{\pm}(x)|\le C$ also. Recall that we of course have that $|R_+|=|R_-|$, but $r_{\pm}$ were obtained
from certain weak limits, so we cannot be sure if these quantities satisfy the same identity.

Write $R_{\pm}-1=\rho_{\pm} e^{i\varphi}$, with $\pi/2<\varphi<3\pi/2$. Note that both numbers indeed have the same
phase $\varphi$ by \eqref{3.6a}. From \eqref{3.6b}, we know that $\rho_{\pm}\ge\rho_0$, where
$\rho_0=|z_0|$ is defined as the absolute value of the smaller of the (usually two) points of intersection of the ray
$te^{i\varphi}$, $t\ge 0$ with the circle $\{ z-1: |z|=C\}$. Also, let $\rho_1=|z_1|\ge \rho_0$ be the absolute value of the other
point of intersection.
We are trying to show that $|R_+|\le C$ or, equivalently, $\rho_+\le\rho_1$.

By elementary geometry, if we compute the points where $\{te^{i\varphi}: t\ge 0\}$ and $\{ z-1 : |z|=C\}$ intersect, we find that
\begin{equation}
\label{3.7}
\rho_j = -\cos\varphi \pm \sqrt{C^2-\sin^2\varphi} \quad\quad (j=0,1);
\end{equation}
more precisely, we obtain $\rho_1$ if we take the plus sign (notice that $-\cos\varphi>0$ for our range
of $\varphi$'s). Notice also that only those values of $\varphi$ for which the ray $te^{i\varphi}$, $t\ge 0$,
intersects the circle $\{ z-1: |z|=C\}$ are compatible with \eqref{3.6a}, \eqref{3.6b}, and this implies that $|\sin\varphi|\le C$.

Next, observe that
\[
-\frac{\overline{R_+}}{R_-} = -\frac{\overline{g}}{g} = \frac{\overline{2ig}}{2ig} = e^{-2i\varphi} .
\]
The last step is by \eqref{3.6a}. It follows that
\[
R_- = -\overline{R_+}e^{2i\varphi} = -\left( 1+\rho_+ e^{-i\varphi} \right) e^{2i\varphi} ,
\]
so $R_- -1=(-\rho_+-2\cos\varphi)e^{i\varphi}$, and this yields
\[
\rho_- = -2\cos\varphi -\rho_+\equiv f(\rho_+) .
\]
From \eqref{3.7}, we see that $f(\rho_1)=\rho_0$,
so we would obtain smaller values $f(\rho_+)<\rho_0$ if we had $\rho_+>\rho_1$, but we also know that $\rho_-\ge\rho_0$, so this
can't happen. We have proved that $\rho_+\le\rho_1$, as desired. \hfill $\square$
\begin{proof}[Proof of Lemma \ref{L3.1}]
The set under consideration can be described as
\[
\{ w=x+iy: 1-C\le e^x \le 1+C, \varphi(x)\le y \le 2\pi - \varphi(x) \} ,
\]
where $\varphi(x)\in (\pi/2,3\pi/2)$ is the angle determined by that ray $te^{i\varphi(x)}$, $t\ge 0$, that intersects the
circle $\{ z-1 :|z|=C\}$ for the first time at $t=e^x$. We want to show that the complementary angle $\psi(x)=\pi-\varphi(x)$
is a concave function of $x$ in the range indicated.

Elementary geometry shows that
\[
\cos\psi(x) = \frac{1}{2} (e^x + Ae^{-x}) ,
\]
with $A=1-C^2>0$. From this, we learn that $(\cos\psi)''=\cos\psi$, so
\[
-\psi'' \sin\psi =(1+\psi'^2)\cos\psi .
\]
Since $0\le\psi<\pi/2$, this implies that $\psi''\le 0$, as desired.
\end{proof}
\section{Reflectionless limits}
As mentioned in the introduction, if we only apply (let's say: left) shifts to (let's say) a Jacobi matrix,
then any limit not only has a $\Sigma_{ac}$ that contains the original $\Sigma_{ac}$ (which is constant
along the approximating sequence, by Propositions \ref{P1.1}, \ref{P1.3}, \ref{P1.2}), but it satisfies the
much stronger condition of being reflectionless on this set \cite[Theorem 1.4]{Remac}.

We now ask the following perhaps somewhat artificial, but not easy to resist question: What is the most
general version of this result that can be obtained by running a sufficiently abstract version of the original proof,
using the setup of the present paper? To make this more concrete, let's define:
\begin{Definition}
\label{D4.1}
A transformation $H_1\mapsto H_2$ of type TM is said to be of type $\mathcal S$ if the associated $T(z)$ can be chosen so that
we have that $T(z)w\in\C^+$ whenever $w,z\in\C^+$. If, in addition, $T(0)=1$, then we say that the transformation is of type $\mathcal S_0$.
\end{Definition}
In other words, in addition to the properties from Definition \ref{D1.1}, we now also demand that the linear fractional
transformation $w\mapsto T(z)w$ is a Herglotz function for all $z\in\C^+$ (equivalently, $z\mapsto T(z)w$ is Herglotz
for all $w\in\C^+$). This condition played a key role in the proof of
\cite[Theorem 1.4]{Remac}, and this is our reason for introducing it here. Abusing notation, we will also write $T\in\mathcal S$.

The relation of the restricted class $\mathcal S_0$ to $\mathcal S$ is easily clarified: if $T(z)\in\mathcal S$ and we set
$A=T(0)\in SL(2,\R)$, then both $T(z)A^{-1}$ and $A^{-1}T(z)$ will be in $\mathcal S_0$.
Conversely, if $A\in SL(2,\R)$ and $T_0(z)\in\mathcal S_0$, then $AT_0, T_0A\in\mathcal S$.
All this is immediate from the definitions if we recall that an $A\in SL(2,\R)$ induces an automorphism of $\C^+$.

We will indeed be able to generalize \cite[Theorem 1.4]{Remac} to this setting. However, nothing essentially new is obtained.
Rather, the transformations of type $\mathcal S_0$ are precisely the left shifts, and if we work with $\mathcal S$, then this
is perhaps followed (or preceded, if we prefer) by
an $SL(2,\R)$ transformation.
\begin{Theorem}
\label{T4.1}
$T(z)\in\mathcal S_0$ precisely if there exist $L\ge 0$ and a trace-normed coefficient $H(x)$ on $[0,L]$ so that
$T(z)=T(L,z)$ where $T(x,z)$ solves
\[
JT'(x,z) = zH(x)T(x,z), \quad T(0,z)=1 .
\]
\end{Theorem}
So this says that the suspense built at the beginning of this section was completely unjustified:
we are not getting beyond the class of shift maps with this attempt at increasing the scope of \cite[Theorem 1.4]{Remac}.
We will obtain a rather general version of this result, though, which we'll state and prove below.

Theorem \ref{T4.1} follows from the characterization of the transfer matrices of canonical systems
that is due to Potapov \cite{Pota} and de~Branges \cite[Theorem 37]{dB}.
Namely, $T(z)=T(L,z)$ for a $T(x,z)$ as above precisely if $T(z)$ is entire with values in
$SL(2,\C)$ and $T(x)\in SL(2,\R)$ if $x\in\R$, $T(0)=1$ and
\begin{equation}
\label{Jin}
-i \left( T^*(z)J T(z) - J \right) \ge 0
\end{equation}
for $z\in\C^+$. If we compare these conditions with the ones from Definition \ref{D4.1}, we see that in order to establish
Theorem \ref{T4.1}, we must show that in the presence of the other conditions just listed, \eqref{Jin} is equivalent to
the Herglotz property from Definition \ref{D4.1}. This is known, but we sketch a proof anyway for the reader's convenience.
The claim boils down to showing the following:
\begin{Lemma}
Let $M\in SL(2,\C)$. Then $w\mapsto Mw$ is a Herglotz function precisely if
\begin{equation}
\label{Min}
-i (M^* JM - J)\ge 0 .
\end{equation}
\end{Lemma}
\begin{proof}
If \eqref{Min} is assumed, then $\textrm{Im}\, Mw>0$ follows at once from the observation that
\begin{equation}
\label{4.2}
\frac{1}{2i} (\overline{w}, 1) J \begin{pmatrix} w \\ 1 \end{pmatrix} = \textrm{\rm Im}\, w .
\end{equation}
So assume now that $Mw$ is Herglotz. Then $M$ maps $\R_{\infty}$ to a circle (possibly a line parallel to $\R_{\infty}$ or even
$\R_{\infty}$ itself, if $M\in SL(2,\R)$) in $\C^+$. We shift so that
the bottom point of this circle touches the real line at $0$ after shifting. In other words, we choose $z\in\C^+\cup \{ 0\}$ so that
$M_1=\left( \begin{smallmatrix} 1 & -z \\ 0 & 1 \end{smallmatrix} \right) M$ is still a Herglotz function, but now $M_1 t=0$
for some $t\in\R_{\infty}$. We also move $t$ to the origin, that is, we introduce
\[
P = \begin{pmatrix} 1 & -z \\ 0 & 1 \end{pmatrix} M\begin{pmatrix} 1 & t \\ 0 & 1 \end{pmatrix} .
\]
This of course assumes that $t\not=\infty$; the case $t=\infty$ is similar and is left to the reader. So we have that
\[
M= \begin{pmatrix} 1 & z \\ 0 & 1 \end{pmatrix} P \begin{pmatrix} 1 & -t \\ 0 & 1 \end{pmatrix} \equiv APB .
\]
It suffices to show that $-iC^*JC\ge -iJ$ for the individual factors $C=A,P,B$. This is obvious for $B$ (with equality,
since $B\in SL(2,\R)$) and $A$ (by a calculation), so we can focus on $P$.

Since $P0=0$, we have that
\[
P = \begin{pmatrix} a & 0 \\ b & a^{-1} \end{pmatrix} ,
\]
with $a\not= 0$. Write $w=x+iy$ and compute
\begin{equation}
\label{4.1}
\frac{1}{2i} (\overline{w}, 1) P^*JP \begin{pmatrix} w \\ 1 \end{pmatrix} = Ay^2 + By + Ax^2 + Cx \equiv F(x,y)
\end{equation}
with
\[
A = \textrm{\rm Im}\, a\overline{b}, \quad B = \textrm{Re}\, a/\overline{a}, \quad C = \textrm{\rm Im}\, a/\overline{a} .
\]
We are interested in the range $x\in\R$, $y\ge 0$. By \eqref{4.2}, the left-hand side of \eqref{4.1} computes a positive multiple of $\textrm{Im}\, Pw$,
so we know that $F(x,y)\ge 0$, and we wish to show that in fact $F(x,y)\ge y$. From the condition that $F\ge 0$, we infer that
$A\ge 0$, $C=0$, so $a/\overline{a}=\pm 1$, that is, either $a=id$ or $a=d$ for some $d\in\R$. The first
case would make $B=-1$, but then $F(0,y)<0$ for small $y>0$, so this is impossible. Hence $a\in\R$ and $B=1$, so
\[
F(x,y) = Ay^2 + y + Ax^2 \ge y ,
\]
as desired.
\end{proof}
\begin{Theorem}
\label{Trefl}
Suppose that $H_n\mapsto H_{n+1}$ is of type $\mathcal S$ and the sequence $H_n$ approaches a
limit $H$ in the sense that $m_{\pm}(z; H_n)\to m_{\pm}(z;H)$
locally uniformly on $\C^+$. Then either:\\
(1) $H$ is reflectionless on $\Sigma_{ac}(H_1)$, or\\
(2) $T_n(z)\cdots T_1(z)\to T(z)$ locally uniformly on $\C$ for some $T\in\mathcal S$ (see comments below).
\end{Theorem}
An equivalent way of stating (1) is to say that $R(x;H)=0$ on $\Sigma_{ac}(H_1)$. This is the conclusion we were hoping for;
in case (2), the whole sequence $H_n$ just approaches the result of mapping $H_1$ by a single shift type map.
Actually, a few comments on the precise formulation of part (2) are in order: Observe that if $H_1\mapsto H_2$
is a given transformation of type $\mathcal S$, then the corresponding matrix $T(z)$ is determined up to a sign. Indeed,
if we again write $T=AT_0$, $T_0\in\mathcal S_0$, $A\in SL(2,\R)$, then $A$ as a linear fractional transformation can be
found from \eqref{defphi}. This gives $A$ as a matrix, up to a sign ($A$ and $-A$ generate the same linear fractional transformation).
We can then see from Theorem \ref{T4.1} that (the matrix) $T_0(z)$ is uniquely determined.
Thus the $T_j(z)$ from case (2) are well defined, up to signs, and the claim really is that after suitably choosing
these signs, we will have the asserted convergence of the matrices.

So in case (2),
we obtain that $H(x)$ is $H_1(x-a)$, possibly followed by an $SL(2,\R)$ transformation. In particular $|R(x,H)|=|R(x,H_1)|$.
Obviously, this second case cannot be ruled out in general; for instance, we could be dealing with the sequence
of shifts $H(x)\mapsto H(x-2^{-n})$.

To prove Theorem \ref{Trefl}, we follow the proof of \cite[Theorem 1.4]{Remac} that was given in this reference as closely as
possible. I will not try to summarize the whole argument here. Rather, I'll focus right away on the one place where new ingredients are needed.
Namely, Lemma A.2 and Corollary A.3 of \cite{Remac} were established by a direct calculation that depended on the specific form
of transfer matrices of Jacobi operators. In our more general setting, we will use a function theoretic version of the Weyl theory
of nested disks for $m$ functions of problems with variable endpoints.

Let me introduce some notation.
Let $T_n(z)\in\mathcal S$. We abbreviate
\[
P_n(z) = T_n(z)\cdots T_1(z) .
\]
Observe that $P_n\in\mathcal S$ also. For fixed $z\in\C^+$ and $n\ge 1$, the set $\{ P_n(z)w : w\in\C^+ \}$ is a disk in the upper half plane
(possibly degenerate, that is, of the form $\{x+iy: y>y_0\}$),
and we let $R_n(z)$ be its hyperbolic diameter:
\[
R_n(z) = \sup_{w,w' \in\C^+} \gamma \left( P_n(z)w, P_n(z)w' \right) \in (0,\infty]
\]
As explained, $\gamma$ denotes hyperbolic distance in the upper half plane. $R_n(z)=\infty$ is possible if the
disk $\{ P_n(z)w\}$ touches $\R_{\infty}$ or is a half plane. Since holomorphic self-maps of $\C^+$ (aka Herglotz functions)
decrease hyperbolic distance, $R_n(z)$ is decreasing in $n$ for fixed $z\in\C^+$. Thus
\[
R(z) = \lim_{n\to\infty} R_n(z) \in [0,\infty]
\]
exists for all $z\in\C^+$.

By normal families, given a sequence
of Herglotz functions, we can always pass to a locally uniform limit on a subsequence. In particular, we can do
this for a sequence of the form $P_n(z)w_n$, with $w_n\in\C^+$. Any such limit will either
be a Herglotz function itself, or it will be identically equal to $a$ for some $a\in\R_{\infty}$.
\begin{Theorem}
\label{TWeyl}
Suppose that for some $w_0,z_0\in\C^+$, the sequence $P_n(z_0)w_0$ has a subsequence that converges to a $\zeta\in\C^+$. Then either:\\
(1) $R(z)\equiv 0$ on $z\in\C^+$, or\\
(2) $R(z)>0$ for all $z\in\C^+$.\\
In the first case, $R_n(z)\to 0$ locally uniformly on $\C^+$.
\end{Theorem}
This can be viewed as a function theoretic version of Weyl's limit point/limit circle alternative. Note, however,
that our disks $\{ P_n(z)w: w\in\C^+\}$ move around and are not nested.
\begin{proof}
I will, as expected, work on the upper half plane in this proof, but our intuition is actually better served
if we imagine our functions taking values in the disk model instead, with its hyperbolic metric.

For $w_n,w'_n\in\C^+$, consider the Herglotz functions $P_n(z)w_n$, $P_n(z)w'_n$.
As explained above, these will approach limits $F(z), G(z)$ along suitable subsequences.
If we also take $w_n\not= w'_n$, then $P_n(z)w_n\not=
P_n(z)w'_n$ for all $z\in\C^+$, so by Hurwitz's Theorem, either $F-G$ is zero free or $F\equiv G$.

Now suppose that
$R(z_0)=0$ for some $z_0\in\C^+$. Then $F(z_0)=G(z_0)$, so $F\equiv G$ for any two limits $F$, $G$.
If we had $R(z_1)>0$ for some $z_1\in\C^+$, then we could pick $w_n,w'_n$ so that $\gamma(P_n(z_1)w_n,P_n(z_1)w'_n)\ge c >0$,
but we also just saw that these two sequences must have the same limit if we make them both converge by passing to a subsequence.
This is only possible if in fact $P_{n_j}(z_1)w\to a\in\R_{\infty}$ for all $w\in\C^+$; here, $a$ is independent of $w$.
However, this condition implies that $P_{n_j}(z)w\to a$ for all $z\in\C^+$. In this whole argument, we are free to
choose $n_j$ as a subsequence of the sequence from the hypothesis of the Theorem. So this scenario is not possible, and
thus indeed $R\equiv 0$ as soon as $R(z_0)=0$.

The locally uniform convergence of $R_n$ in this case follows by essentially the same argument.
Fix a sequence $n_j\to\infty$ as in the hypothesis, that is, $P_{n_j}(z_0)w_0\to\zeta\in\C^+$.
Now if we had $R_n(z'_n)\ge 2\epsilon >0$ for arbitrarily large $n$ and certain $z'_n$'s drawn from a compact
set $K\subset\C^+$, then, since $R_n$ is decreasing, we could also arrange that
$R_{n_j}(z_j)\ge 2\epsilon$ for certain arbitrarily large $j$. We could then find $w_j,w'_j\in\C^+$
so that $\gamma(P_{n_j}(z_j)w_j, P_{n_j}(z_j)w'_j)\ge \epsilon>0$,
even though both sequences approach the same limit, uniformly on $K$ (after passing to another subsequence).
This leaves only one possible excuse: the common limit must lie in $\R_{\infty}$, but this contradicts the choice of the sequence $n_j$,
which makes sure that all limits are genuine Herglotz functions.
\end{proof}
Let us now return to Theorem \ref{Trefl}. First of all, the hypothesis of Theorem \ref{TWeyl} is clearly satisfied here:
we know that $P_n(z)m^{(1)}_-(z)\to m_-(z)\in\C^+$.

We wish to show that cases (1), (2) of Theorem \ref{TWeyl} lead
to (1) and (2) of Theorem \ref{Trefl}, respectively.

If we are in case (1) of Theorem \ref{TWeyl}, then the additional statement that $R_n(z)\to 0$ \textit{uniformly }on compact
subsets of $\C^+$ can now take the role of Corollary A.3 of \cite{Remac} (which was used in the original proof).
The argument now proceeds exactly as in \cite{Remac}, with no further modifications necessary. I'll leave the matter at that.

Now consider case (2) (of Theorem \ref{TWeyl}). By the observations made above, a $T\in\mathcal S$ can be written as
either $T=AT_0$ or $T=T_1A$, with $A=T(0)\in SL(2,\R)$ and $T_j\in\mathcal S_0$. This means that in the product $P_n$, we can
move the constant $SL(2,\R)$ factors all the way to the left and write
\[
P_n(z) = B_n S_n(z)\cdots S_1(z) ,
\]
with $B_n\in SL(2,\R)$ and $S_j\in\mathcal S_0$. Observe also that the $S_j$, $1\le j\le n$, don't change if we add new factors on the left and
consider $P_k$ with $k>n$. By Theorem \ref{T4.1}, each $S_j$ is a transfer matrix of a canonical system on an interval $I_j$.
The product $S_n\cdots S_1$ then is the transfer matrix of the concatenation of these canonical systems. More specifically, arrange this so that
$S_j$ is the transfer matrix across $(x_{j-1},x_j)$. I now claim that $\lim_{j\to\infty} x_j<\infty$. This will follow from the important result
of de~Branges \cite[Theorem 42]{dB} that a trace-normed canonical system on a half line is in the limit point case at the singular endpoint.
See also \cite{AchaLPC} for a transparent alternative proof.

So assume, to obtain a contradiction, that $x_j\to\infty$. We shift the intervals so that we now have canonical systems on $[-x_j,0]$, and
we want to consider limits of these problems; these will be canonical systems on a left half line $(-\infty,0]$.
We are taking these limits with respect to a metric that is built in a similar way as the metric
on (Schr\"odinger operator) potentials that was used in \cite{Remcont}. Two canonical systems are close if their coefficient functions
almost agree (in a weak sense, not pointwise) on a sufficiently long initial interval $[-L,0]$; the metric pays little attention to what happens far out.
We will not spell out the details here; the important features
are that this gives a compact space of canonical systems, and convergence with respect to this metric is equivalent to the
locally uniform convergence of the $m$ functions. (In fact, this property could almost serve as the definition.)

Fix any limit point $H_0(x)$, $x<0$. By standard Weyl theory and what we just reviewed, any limiting function of the type
\begin{equation}
\label{4.11}
M(z)=\lim P_{n_j}(z)w_{n_j}
\end{equation}
is a half line $m$ function of this problem. By the result just quoted, there can be only one
such function. However, we are currently assuming that $R(z)>0$ on $\C^+$, and this makes sure
that we can get more than just one function from \eqref{4.11}. Indeed it suffices to fix $z_0\in\C^+$,
then pick $w_n,w'_n$ so that $\gamma(P_n(z_0)w_n,P_n(z_0)w'_n)\ge c>0$ and pass to limits along
a subsequence.

This contradiction shows that the collection of intervals $(x_{j-1},x_j)$ is of bounded total length. So there exist $L\ge 0$ and
a trace-normed $H(x)$ on $(0,L)$ so that $P_n(z)=B_n T(x_n,z)$, where $T(x,z)$ is the transfer matrix of this system across $(0,x)$,
and $0=x_0\le x_1 \le \ldots$, $x_n\to L$. So in particular, we have that $T(x_n,z)\to T(L,z)$ locally uniformly on $\C$,
and this matrix is in $\mathcal S_0$.

It remains to show that with suitable choices of ($n$ dependent) signs, we will also obtain that $\pm B_n\to B$ for some
$B\in SL(2,\R)$. Recall that
\begin{equation}
\label{4.21}
B_nT(x_n,z)m_-^{(1)}(z)\to m_-(z) .
\end{equation}
By normal families again, $B_n$ (as a sequence of linear fractional
transformations) converges locally uniformly on $\C^+$ along suitable subsequences. The limit cannot be identically equal
to an $a\in\R_{\infty}$, by \eqref{4.21}, so it is a linear fractional transformation $B\in SL(2,\R)$ itself. Moreover,
\eqref{4.21} identifies $B$ (as a map) uniquely, so there is no need to pass to a subsequence. For the matrices, this means
that $\pm B_n\to B$, as desired (note that we originally derived the convergence of the \textit{maps, }not the matrices).

\end{document}